\documentclass[12pt]{amsart}

\usepackage{amsfonts, amssymb}
\usepackage{epsf}
\usepackage[hug,balance,dpi=1200]{diagrams} 

\newtheorem{theorem}{Theorem}[section]
\newtheorem{lemma}[theorem]{Lemma}

\newtheorem{corollary}[theorem]{Corollary}

\newtheorem{prop}[theorem]{Proposition}

\numberwithin{equation}{section}

\theoremstyle{definition}

\theoremstyle{remark}
\newtheorem{remark}[theorem]{Remark}

\renewcommand{\mathbb}{\mathsf}
\renewcommand{\mathbf}{\mathsf}
\begin{document}

\title[the tensor product of Cuntz semigroups]{The Cuntz semigroup of the tensor product of C*-algebras}
\author{Cristian Ivanescu}
\author{Dan Ku\v cerovsk\'y}
\address[1]{Department of Mathematics and Statistics, Grant MacEwan University, Edmonton, Alberta, T5J 4S2, e-mails:  ivanescu@ualberta.ca, ivanescuc@macewan.ca}
\address[2]{Department of Mathematics and Statistics, University of New Brunswick, Fredericton, New Brunswick, E3B 5A3, e-mail: dkucerov@unb.ca}


\date{Received December 8, 2014}

\commby{}

\begin{abstract}{
We calculate the Cuntz semigroup of a tensor product C$^*$-algebra $A \otimes A$. We restrict our attention to C$^*$-algebras which are unital, simple, separable, nuclear, stably finite, $\mathcal{Z}$-stable, satisfy the UCT, with finitely generated $K_0(A)$ and have trivial $K_1(A)$.} \\ \\  \\

 { On calcule le semigroupe de Cuntz d'une C$^*$-alg\`ebre produit tensoriel $A \otimes A$. On consid\`ere seulement les C*-alg\`ebres simples, s\'eparable, nucl\'eaires, $\grave{\mathrm{a}}$ \'el\'ement unit\'e, stablement finies, $\mathcal{Z}$-stables, satisfaisant au UCT, dont le groupe $K_0(A)$ est de type fini, et dont le groupe $K_1(A)$ est trivial.}
\end{abstract}

\maketitle



\newpage

\section{Introduction}
The Cuntz semigroup has been studied since the late seventies but only recently has it proven to be an important invariant for $C^*$-algebras. First, in the early 2000s, M.~Rordam and A.~Toms constructed examples of $C^*$-algebras that appeared to be counterexamples to the Elliott conjecture. Shortly afterwards, Toms realized that the Cuntz semigroup distinguishes some of the newly constructed algebras; hence, the Cuntz semigroup could be added to the Elliott invariant. Toms's discovery obviously prompted major questions, such as:  ``What is the range of the Cuntz semigroup?'' or ``What is the relation between the Cuntz semigroup and the Elliott invariant?'' or ``What are the properties of the Cuntz semigroup?''\\
{\indent} In this paper we propose to study one property of the Cuntz semigroup, namely, how the Cuntz semigroup of the tensor product, $A\otimes A$, of two identical copies of the C*-algebra $A$ relates to the Cuntz semigroup of $A$. 
It is well known that the tensor product of two positive elements is still a positive element. This property allows us to define a natural tensor product map from $A^{+} \otimes A^{+}$ to $(A\otimes A)^{+}$. The usual interpretation of $A^{+} \otimes A^{+}$ is as a subset of the usual tensor product of (nuclear) $C^*$-algebras. However, defining maps at the level of Cuntz semigroups requires defining tensor products of Cuntz semigroups, which are \textit{a priori} tensor products of abelian semigroups. Hence we must consider semigroup tensor products, discussed below.  Our approach to tensor products of Cuntz semigroups is to first take an algebraic tensor product of abelian semigroups and then to take a completion with respect to a suitable topology. See \cite[para. 6.L]{kelley} for more information on topological completions. The basic reason for introducing completions is that if we use only algebraic tensor products we can   obtain surjectivity results only in very limited situations, such as the finite-dimensional case. In the first three sections of this paper we work with the algebraic tensor product, and we use the term  ``dense range'' for results from which we later obtain surjectivity as a corollary, after taking a completion.  We consider completions in the last section of the paper. 

  As defined by Grillet \cite{grillet}, the tensor product of two abelian semigroups is constructed by forming a free abelian semigroup and passing to the quotient by the relations $(a + a')\otimes b=(a\otimes b)+(a'\otimes b)$ and $(a\otimes b')+(a \otimes b)=a\otimes (b+b')$. This definition is equivalent \cite{grillet} to the definition by a universal property. Stating the universal product definition for 
a family $(A_i )_{i\in I}$ of semigroups, we first say that a mapping 
$s$ of the Cartesian product of semigroups $\prod A_i $ 
into a semigroup $C$ is $I$-linear 
if the mapping is a semigroup homomorphism in each variable separately.  Then, if an $I$-linear  
mapping $t$ of $\prod A_i $ into a semigroup $T$ has the property that, for any 
$I$-linear mapping $s$ of $\prod A_i $ into some semigroup $C$,
there exists a unique homomorphism $u$ of $T$ into $C$ such that 
$s = u\circ  t,$ then we call the pair $(t,T),$ and also the semigroup $T,$ a tensor product of the family $(A_i )_{i\in I}.$  

It is well known that not every positive element of a tensor product can be written as a tensor product of positive elements, even after allowing sums. Thus, the naive tensor product map from $A^{+} \otimes A^{+}$ to $(A\otimes A)^{+}$ is in general not surjective. It seems interesting that, as we shall see, in some cases this map becomes surjective if we pass to Cuntz equivalence classes. 

In a recent paper \cite{APT14}, the question of determining surjectivity, at the level of Cuntz semigroups, of the natural tensor product map is posed; and left as an open problem. In that paper, the authors state that surjectivity does hold in the cases of AF algebras and $O_\infty$-stable algebras. This is not surjectivity at the level of algebraic tensor products, rather it  is  with respect to a particular choice of Cuntz semigroup tensor product introduced in \cite{APT14}, called the Cuntz category tensor product.  We will consider the case of simple, separable, unital, stably finite, nuclear, $\mathcal{Z}$-stable C$^*$-algebras, with finitely generated $K_0$ group, trivial $K_1$-group and satisfying the UCT, and we show that the image of the natural tensor product map on algebraic elements is dense in the sense that it becomes surjective after passing to a completion. We consider completions with respect to several different possible topologies, the coarsest of these being given by pointwise suprema, as will be explained in the last section of the paper. We use \cite{rorSR1} to deduce that algebras in the abovementioned class have stable rank one. The stable rank one property and its consequences are used several times in our proofs.

Brown, Perera and Toms, \cite{BPT08}, showed an important representation result for the original version of the Cuntz semigroup. This result was extended to the nonunital case, using the stabilized version of the Cuntz semigroup, by Elliott, Robert and Santiago, \cite{ERS}, and with more abstract hypotheses by Tikuisis and Toms,  \cite{TT}. Their results (see Theorems \ref{th:representations} and \ref{th:representations}) imply that for certain simple exact C$^*$-algebras, a part of the Cuntz semigroup is order isomorphic to an ordered semigroup of lower semicontinuous functions defined on a compact Hausdorff space.

\section{The Cuntz semigroup}
Let $A$ be a separable C*-algebra. For positive elements $a,b \in A \otimes \mathbb{K}$, we say that $a$ is Cuntz subequivalent to $b,$ and write $a\preceq b$, if  $v_{\textit{n}}^{}bv_{\textit{n}}^{*}\rightarrow a$ in the norm topology, for some sequence $(v_{\textit{n}})$ in $A \otimes  \mathbb{K}$. We say that $a$ is Cuntz equivalent to $b$ and write $a \sim b$ if $a \preceq b$ and $b\preceq a$. Denote by $Cu(A)$ the set of Cuntz equivalence classes of the  positive cone of $A \otimes \mathbb{K}$, i.e. $Cu(A)= (A \otimes \mathbb{K})^+/_{\sim}$. The order relation $a \preceq b$ defined for the positive elements of $A \otimes \mathbb{K}$ induces an order relation on $Cu(A)$: $[a] \leq [b]$ if $a \preceq b$, where $[a]$ denotes the Cuntz equivalence class of the positive element $a$. Note (\textit{cf.} \cite[page 151]{RW}) that this order relation does not need to be the algebraic order with respect to the addition operation defined by setting $[a]+[b]:=[a'+b']$, where $a'$ and $b'$ are orthogonal positive elements. It turns out that in a stabilization we can always find such orthogonal representatives, \textit{i.e.,} in $(A \otimes \mathbb{K})^+$ we have $a \sim a'$, $b \sim b'$ with $a'b'=0$. Moreover, the choice of the orthogonal representatives does not affect the Cuntz class of their sum. So the ordered set $Cu(A)$ becomes an abelian semigroup, under an addition operation that is sometimes called Brown-Douglas-Fillmore addition \cite{BDF}.
If $A$ is unital, we denote by $T(A)$ the simplex of tracial states. 
By $V(A)$ we denote the projection semigroup defined by the Murray von Neumann equivalence classes of projections in $A \otimes\mathbb{K}$. The order structure on $V(A)$ is defined through  Murray-von Neumann comparison of projections.


\subsection{Representations of the Cuntz semigroup} 

 Brown, Perera and Toms's representation result \cite{BPT08} for the Cuntz semigroup is as follows:
\begin{theorem} Let $A$ be a simple, separable, unital, exact, stably finite $\mathcal{Z}$-stable C$^*$-algebra. Then there is an order preserving isomorphism of ordered semigroups, \label{th:representations}
 $$W(A) \cong V(A) \sqcup Lsc(T(A), (0, \infty)).$$
\end{theorem} 
In the statement of the above theorem, $W(A)$ is the original definition of the Cuntz semigroup, \textit{i.e.}, $W(A)=M_{\infty}(A)^{+}/_{\sim},$ and $Lsc(T(A),(0,\infty))$  denotes the set of lower semicontinuous, affine, strictly positive functions on the tracial state space of the unital C*-algebra $A$. Addition within $Lsc(T(A),(0,\infty))$ is done pointwise and order is defined through pointwise comparison, as is usual for functions. For $[p] \in V(A)$ and $f \in Lsc(T(A),(0,\infty])$, addition is defined by
$$[p]+f:=\hat{[p]}+f \in Lsc(T(A),(0,\infty)\,),$$ where $\hat{[a]}(\tau)= \lim\limits_{n \rightarrow \infty}\tau(a^{1/n}), \tau \in T(A),$ which reduces to $\tau(a)$ when $a$ is a projection. The order relation is given by:
  $$[p] \leq f \;\mathrm{if}\;\hat{[p]}(\tau) < f(\tau)\;\mbox{ for all } \tau \in T(A),$$
  $$ f \leq [p]\;\mathrm{if}\;f(\tau) \leq \hat{[p]}(\tau) \mbox{ for all } \tau \in T(A). $$
 
 Elliott, Robert and Santiago's  representation result \cite{ERS} is very similar, and uses the stabilized Cuntz semigroup.  In this result, the functions that appear may take infinite values and the algebras are not necessarily unital. Since we restrict our attention to the case of unital algebras, the domain, $T(A),$ can be taken to be a compact simplex, which in turn gives a simplified version of their result:
 \begin{theorem} Let $A$ be a simple, separable, unital, exact, stably finite ${\mathcal Z}$-stable C*-algebra. Then there is an order preserving isomorphism of ordered semigroups,
 $$Cu(A) \cong V(A) \sqcup Lsc(T(A),(0, \infty]),$$
 where $Lsc(T(A),(0,\infty])$ will denote the set of lower semicontinuous, possibly infinite, affine, strictly positive functions on the tracial state space of a unital C*-algebra $A$. Within $Lsc(T(A), (0, \infty])$ addition is pointwise and pointwise comparison is used. For $[p] \in V(A)$ and $f \in Lsc(T(A),(0,\infty])$, addition is given by
$$[p]+f:=\hat{[p]}+f \in Lsc(T(A),(0,\infty]),$$ where $\hat{[a]}(\tau)= \lim\limits_{n \rightarrow \infty}\tau(a^{1/n}), \tau \in T(A),$ which reduces to $\tau(a)$ when $a$ is a projection. The order relation is given by:
  $$[p] \leq f \;\mathrm{if}\;\hat{[p]}(\tau) < f(\tau)\;\mbox{ for all } \tau \in T(A),$$
  $$ f \leq [p]\;\mathrm{if}\;f(\tau) \leq \hat{[p]}(\tau) \mbox{ for all } \tau \in T(A). $$
 \label{th:representations2}
\end{theorem}
In the proof of the above theorems, a semigroup map $i:Cu(A) \longrightarrow Lsc(T(A),(0, \infty])$ is defined,
$i([a])(\tau)=d_{\tau}(a),$ with $d_{\tau}$ to be explained later.

  These theorems show that the Cuntz semigroup, say $Cu(A)$, is the disjoint union of the semigroup of positive elements coming from projections in $(A\otimes \mathbb{K})^+$, denoted $V(A)$, and the set of lower semicontinuous, affine, strictly positive, functions on the tracial state space of $A$, denoted by $Lsc(T(A), (0, \infty])$. In \cite{BPT08}, the elements of the Cuntz semigroup that correspond to lower semicontinuous, affine, strictly positive, functions on the tracial state space are termed \textit{purely positive} elements. In general, the set of purely positive elements does  not form an object in the Cuntz category. To see this, consider an element $x$ with spectrum $[\epsilon,1]$ in a C*-algebra of stable rank 1. The increasing sequence $(x-\frac{1}{n} )_{+}$ is at first purely positive, but has a supremum that is projection-class.	
	
Theorem \ref{th:representations2} implies that the subsemigroup of purely positive elements of the Cuntz semigroup, for certain C*-algebras $A$, is isomorphic to the semigroup  $Lsc(T(A),(0, \infty]).$ 
 The convex structure of the space of tracial states, $T(A)$, makes it a Choquet simplex when $A$ is unital \cite{EH}, metrizable when $A$ is separable. 

 We will need a result about lower semicontinuous functions on  metrizable Choquet simplices. 
\begin{prop} Let $S$ be a compact metrizable Choquet simplex. Then every positive element of $Lsc(S, (0, \infty]),$ bounded or not, is the pointwise supremum of some pointwise nondecreasing sequence of continuous positive functions on $S.$ \label{prop:countable.sup.of.cont}
\end{prop} 
\begin{proof}  Let $S$ be a compact metrizable Choquet simplex. Using Edward's separation theorem inductively shows, see Lemma 6.1 in \cite{ABP}, that every lower semicontinuous positive affine function on the simplex, possibly with infinite values, is the pointwise supremum of a strictly increasing sequence of affine continuous functions without infinite values. Compactness lets us arrange that the functions are everywhere positive, for example, we may replace each $f_n$ by the pointwise supremum of $\{f_n,\epsilon 1\}$ for a suitably small $\epsilon.$  
 \end{proof}


The dual of the Cuntz semigroup is denoted by $D(A),$ and is referred to as the set of all dimension functions.  It is the set of all additive, suprema-preserving, and order preserving maps $d:Cu(A) \rightarrow (0,\infty]$ such that, in the unital case, $d([1])=1.$ If the map on $A^+$ given by $x\mapsto d([x])$ is lower semicontinuous, we say that the dimension function is lower semicontinuous. The lower semicontinuous dimension functions correspond to the 2-quasitraces, by Proposition 2.24 of \cite{BK}; for more information, see \cite{ERS}.
 In the general case, once Theorem \ref{th:representations2} is no longer applicable, the dual space of the Cuntz semigroup is strictly larger than the set of traces, $T(A).$   See  \cite[page 307]{BlH} for an example of a dimension function that is not lower semicontinuous and thus does not come from either a trace or quasitrace. This example arises from a nonsimple and nonunital C*-algebra. We don't know if there is an example coming from a  simple and nuclear C*-algebra. We also note that nonsimple purely infinite algebras  may have a nontrivial Cuntz semigroup, but do not have traces. Thus, their dimension functions do not come from traces. 

 Any (necessarily not exact; see \cite{uffe}) C*-algebra for which the quasi-traces are not all traces will be an example where the states on the Cuntz semigroup do not correspond to the traces. 
 Reviewing the literature dealing with the Cuntz semigroup, the algebraic structure of the Cuntz semigroup is generally the main topic of interest, and the topological structure is hardly ever explicitly mentioned. We mention here some minor but apparently new observations about the topology of the Cuntz semigroup.
We have Hausdorff metrics: $D(A)$ is metrizable when  $A$ is separable, the metric being given by $\sum |d_1(x_k )-d_2(x_k)|2^{-k},$ where $x_k$ is a dense subsequence of the positive part of the unit ball of $A.$ Similarly, the Cuntz semigroup itself has at least a pseudometric, in the presence of separability, of the form  $\sum |d_k(x_1 )-d_k(x_2)|2^{-k},$ where $d_k$ is a dense subsequence of $D(A).$ 
We note that in general there may exist projection-class elements that are equal, under the dimension functions, to purely positive elements. In the stable rank 1 case, it is possible to discriminate between projection-class elements and purely positive elements on the basis of a spectral criterion.   

\subsection{Dimension functions and a conjecture of Blackadar and Handelman} 

We have seen, as in Theorem \ref{th:representations2}, that the map $i$ is useful in describing the order on the Cuntz semigroup. The map $i$ is  $i(a)=d_{\tau}(a),$ where we define $d_{\tau}(a)$ to be an extended version of the rank of $a$:
$d_{\tau}(a)= \lim\limits_{n \rightarrow \infty}\tau(a^{1/n})$, where $\tau$ is a tracial state. This map, $d_{\tau}$, also called a dimension function, is lower semicontinuous as a map from $A^+$ to $[0,\infty],$ possibly taking infinite values, and defines a state on the Cuntz semigroup. 
 In 1982, Blackadar and Handelman conjectured, see \cite{BlH}, that the set of lower semicontinuous dimension functions that come from traces is weakly dense in the set of dimension functions (or states on the Cuntz semigroup). The conjecture is known to be true for a large class of C*-algebras, see \cite[page 426]{Toms09},   that includes the algebras that we propose to study in this paper, namely: simple, unital, stably finite, nuclear, $\mathcal{Z}$-stable C*-algebras, with stable rank one, trivial $K_1$-group and with  the UCT property. We note that by \cite{rorSR1}, the stable rank one property follows from the other properties. From now on, we thus assume that the Blackadar-Handelman conjecture holds.


 Consider the map $t: A^{+} \times A^{+} \rightarrow Cu( A \otimes A)$ defined by
 $$t(a,b)=[a\otimes b].$$

Let us check that the above map $t$  respects Cuntz equivalence.

\begin{lemma} Let $A$ be a $\sigma$-unital C*-algebra. Given positive elements $a,a',b$ in $A$ such that $a' \preceq a$ we have $a \otimes b \preceq a' \otimes b$. 
\end{lemma}
\begin{proof} Let $e_n$ be a countable approximate unit.  Since $a' \preceq a$, choose an $e_n$ such that $e_nae_n^* \rightarrow a'$. We have
\[(e_n \otimes e_n)(a \otimes b)(e_n \otimes e_n)^*-a' \otimes b=
e_n a e_n^* \otimes e_n b e_n^* -a' \otimes b,\]




and so 
 $$||(e_n \otimes e_n)(a \otimes b)(e_n \otimes e_n)^* -  a' \otimes b|| \rightarrow 0$$ 
\end{proof}

If $a \sim a'$ then $a\otimes b \sim a' \otimes b$ by applying the lemma twice, and thus we obtain the Corollary:
\begin{corollary} Consider the map $t: A^{+} \times A^{+} \rightarrow Cu( A \otimes A)$ defined by
 $t(a,b)=[a\otimes b].$ If $a$ and $a'$ are positive elements of $A$ that are Cuntz equivalent, then $t(a,b)=t(a',b).$  
\end{corollary}

\section{Main result}

We begin with a technical lemma that is used in proving our main results. 
 
\begin{lemma} Let $S_1$ and $S_2$ be  compact metrizable Choquet simplices. Let  $F$ be a positive, (bi)affine, continuous finite-valued function on $S_1 \times S_2 .$ The function $F$ can be approximated uniformly from below by a finite sum $\sum_{i,j} a_{ij}f^{(1)}_i(x)f^{(2)}_j(y)$ where the $f^{(k)}_{\textit{i}}$ are continuous affine positive functions on $S_k$ and the $a_{ij}$ are positive scalars. \label{lem:pou.affine}
\end{lemma}
\begin{proof} Suppose first that $S_1 = S_2 =S.$ The affine continuous functions $A(S)$ on the compact Choquet simplex $S$ happen to be a Banach space whose dual is an $L_1$-space, see, \textit{e.g.}, \cite[pg.181]{LL}.  The space $A(S)$ is separable because $S$ is metrizable. We can therefore apply Theorem 3.2 in \cite{LL} to obtain an inductive limit decomposition of $A(S)$ in the form
$$ A(S)=\overline{\bigcup E_n,}$$
where the $E_n$ are finite-dimensional $l_\infty$-spaces and the connecting maps are inclusion maps. We denote the dimension of $E_n$ by $m_n.$ Each  subspace $E_i=\ell_{\infty}^{m_i}$ has a basis, $\{f_{j}\}_{j=1}^{m_i},$ of elements of $A(S_{m_i})^+$ with $\sum_{j} f_{j} =1$ in that subspace. Since the connecting maps in the above inductive limit are inclusion maps, we can inductively choose the bases in such a way that the set of basis functions for $\ell_{\infty}^{m_i}$ is contained in the set of basis functions for $\ell_{\infty}^{m_{i+1}}.$ The union of these sets of basis functions gives an infinite sequence $(f_i)\in A(S)^+$.  Choose a dual sequence $(x_i)\in S;$ thus, $f_i(x_j)$ is 1 if $i=j,$ and is zero otherwise. 
Passing to duals we also obtain a projective limit decomposition of $S$ in the form
\begin{diagram}[small]
S_{m_1} &\lTo & S_{m_2} & \lTo & S_{m_3}  & \lTo & \cdots \\
\end{diagram} where the $S_{m_i}$ are  $m_i$-dimensional simplexes and the maps are affine surjective maps. 

 
Consider the (bi)affine function $F_n (x,y):= \sum_{j,k}^{m_n} a_{jk} f_{j}(x) f_{k}(y)$ where $a_{jk}$ is the given function $F(x,y)$ evaluated at $(x_j,x_k).$ The $F_n$ are an increasing sequence of positive functions that converges pointwise, on a compact space, and Dini's theorem implies uniform convergence. Denoting the  (bi)affine function obtained in the limit by $G,$ we note that $G$ is equal to the given function $F$ at points of the form $(x_i,x_j).$ Because both $G$ and $F$ are already known to be continuous and positive, to prove equality it suffices to show that the $(x_i)$ are (affine) linearly dense in a suitable sense. We recall that the affine span of a subset of an affine space is the set of all finite linear combinations of the points of the subset, with coefficients adding up to 1. When we apply a projection map $p_n\colon S\rightarrow S_{m_n},$ a dimension argument using the basis property shows that the simplex generated by  $(p_n(x_i))$ has maximal dimension; or in other words, the affine span of the $(p_n(x_i))$ includes all of $S_{m_n}.$   It follows from the definition of the topology on a projective limit that the affine span of $(x_i)$ is dense in $S.$   It then follows that $G$ equals $F$ on a dense set, and thus everywhere.  
                 
								We thus obtain finite sums, $\sum_{i,j} a_{ij}f_i(x)f_j(y),$ that approximate the given function $F$ as required,  where the $a_{ij}$ are positive scalars, however, the non-negative functions $f_k$ may have zeros. We can arrange that there are no zeros by first approximating  $F-\varepsilon$ for sufficiently small $\varepsilon>0,$ and then adding small positive multiples of $1$ to the resulting functions $f_k.$ 
								
								The case  $S_1 \not= S_2$ is a straightforward generalization. 
\end{proof}
The next theorem appears to be of interest in the setting of nonunital stably projectionless C*-algebras. In the nonunital case, one has to find a compact simplex as base for the tracial cone: see \cite[Prop. 3.4]{TT}. We will subsequently  build upon the proof of the next theorem to accomodate projection-class elements. We denote the algebraic tensor product of abelian semigroups \cite{grillet} by $\otimes_{alg}.$ Under the hypotheses of the theorem, the Cuntz semigroups are sets of semi-continuous functions, and dense range means that the image is norm-dense among the continuous functions (which are dense in a pointwise sense among the lower-semi-continuous functions in question, by Proposition \ref{prop:countable.sup.of.cont}). 
See Section 4 for more information on the topology that is implied by the density condition that we are using.
 
The next lemma extends, so to speak,  the familiar properties of tensor products of continuous functions to the case of lower semi-continuous functions, by combining  Proposition \ref{prop:countable.sup.of.cont} and Lemma \ref{lem:pou.affine}. 
\begin{lemma} Assume $A$ and $B$ are separable C*-algebras that are such that $Cu(A) \cong  Lsc(T(A), (0, \infty]),$
$Cu(B) \cong  Lsc(T(B),(0, \infty]),$ and $Cu(A\otimes B) \cong Lsc(T(A)\times T(B),(0,\infty]).$ Then the tensor product map $t$ from $Cu(A) \otimes_{alg} Cu(B)$ to  $Cu(A \otimes B)$ has dense range, in the sense that every element of $Cu(A \otimes B)$ can be obtained as the pointwise supremum of some sequence of elements of the range of the map $t.$ \label{th:surjective.function.semigroup} 
\end{lemma}
\begin{proof} When $A$ and $B$ are  non-unital, we use \cite[Prop. 3.4]{TT} to find a compact simplex, denoted $T(A)$ or respectively $T(B),$ as a base for the tracial cone.
 Define $\Phi : Lsc(T(A),(0,\infty]) \otimes_{alg} Lsc(T(B),(0,\infty]) \rightarrow Lsc(T(A)\times T(B),(0,\infty])$, a semigroup morphism that acts on the elements $(f \otimes g)$ of $Lsc(T(A),(0,\infty]) \otimes_{alg} Lsc(T(B),(0,\infty])$ as
$$\Phi(f \otimes g)(x,y)=f(x)g(y),$$ where the right hand side is the pointwise product, $f(x)g(y),$ interpreted as a function on a Cartesian product. We note that the semigroup tensor product on the left is defined \cite[page 270]{grillet} by forming a free abelian semigroup and imposing all possible bilinearity relations: $(a+a')\otimes_{alg} b=a\otimes_{alg} b+a'\otimes_{alg} b$ and  $a\otimes_{alg} b'+a \otimes_{alg} b=a\otimes_{alg} (b+b')$. Evidently these relations do not alter the function obtained on the right hand side.
 
Note that the product of non-negative lower semicontinous functions is a non-negative lower semicontinuous function. The given representations of the Cuntz semigroup(s) evidently intertwine $\Phi$ with the given tensor product map $t.$
%
%

%

We now consider the range of the map $\Phi.$ Let $G$ be a continuous element of $Cu(A\otimes B) \cong Lsc(T(A)\times T(B),(0,\infty]).$ 
We recall that for a unital C*-algebra, the tracial states, $T(A)$, are a Choquet simplex (see
\cite{EH}). In the nonunital case, $T(A)$ is understood to be a compact and simplicial base of the tracial cone (as already mentioned above, see also \cite{TT}). Since $A$ is separable, the simplex is metrizable. By  Lemma \ref{lem:pou.affine} we have an element $F_n$ of the form $\sum_{i}^{N}f_i(x)g_i(y)$ that is bounded above by $G$ and approximates the continuous function $G$ within $\tfrac{1}{n}.$  Since the functions $f_i$ and $g_i$ are positive,  the element $F_n$ is in the image of the map $\Phi.$ The more general case of $G$ being lower semicontinuous follows by taking suprema, using Proposition \ref{prop:countable.sup.of.cont}. 
\end{proof}



\subsection{Beyond the purely positive case}

 The next step is to assume our algebra $A$ has non-trivial projections, so that  $V(A)$, the projection monoid, is non-trivial. \\
  We will make use of the fact that the Cuntz semigroup  is the disjoint union of projection elements from the Cuntz semigroup, denoted $V(A),$ and purely positive elements, denoted $Lsc(T(A),(0,\infty])$. Ara, Perera and Toms, \cite[Proposition 2.23]{APT}, prove that that for algebras of stable rank one, the Cuntz class of a positive element is given by a projection if and only if $\{0\}$ is not in the spectrum or if it is an isolated point of the spectrum.

If $a \in M_{\infty}(A)^+, b \in M_{\infty}(B)^+$ are positive elements then it follows that $a \otimes b$ is a positive element in $M_{\infty}(A \otimes B)^+$. This induces a bilinear morphism from $Cu(A) \times Cu(B)$ to $Cu(A \otimes B)$ which in turn induces a natural Cuntz semigroup map 
$$ t: Cu(A) \otimes_{alg} Cu(B)  \rightarrow Cu(A \otimes B)$$
$$t([a] \otimes [b])=[a \otimes b].$$

In the case that $A=B$ and $A$ is a simple, separable, unital, stably finite, nuclear, ${\mathcal Z}$-stable C*-algebra, satisfies the UCT, with stable rank one and vanishing $K_1$ group, we now show that the map $t$ has  dense range. We begin with a technical lemma on enveloping groups:
\begin{lemma} Let $V$ be a semigroup, and let $G$ denote the formation of the enveloping group. We have $$G(V \otimes V) = G(V)\otimes_Z G(V),$$ where $\otimes$ denotes the tensor product of semigroups, and $\otimes_Z$ denotes the tensor product of abelian groups.
\label{lem:tech}\end{lemma}
\begin{proof} Let us denote by $C(S)$ the greatest homomorphic image of a commutative semigroup $S$. This is just the quotient of $S$ by an equivalence relation:  $x\sim y$ if $x+b=y+b$ for some element $b.$ Following \cite{grillet2}, we define a tensor product of commutative semigroups by $S_1 \otimes^{C} S_2 := C(S_1\otimes S_2).$ It is shown on page 201 of \cite{grillet2} that this is an associative and commutative tensor product of commutative semigroups. By Proposition 3 of \cite{grillet2}, we have that the enveloping group of $V$ is given by
$$G(V) = V\otimes^{C} Z, $$ where $Z$ denotes the group of integers.
But then, by the associative and commutative properties of the tensor product,
$$G(V)\otimes^{C} G(V) = (V\otimes^C V) \otimes^{C} Z. $$
It then follows that 
\begin{eqnarray*}G(V)\otimes^{C} G(V) =& C(V\otimes V) \otimes^{C} Z \\
                                      =& G(C(V\otimes V)).\\
\end{eqnarray*}

The definition of an enveloping group is such that the enveloping group of $C(V\otimes V)$ is the same as the enveloping group of $V\otimes V.$
We have therefore shown that $G(V\otimes V)=G(V)\otimes^{C} G(V).$ However, by Proposition 1.4 of \cite{grillet}, the semigroup tensor product of abelian groups coincides with the usual tensor product of abelian groups, and  since groups are already cancellative, there is then no difference between the tensor products $\otimes$, $\otimes_Z$ and $\otimes^C$ when the factors are both abelian groups. It follows that $G(V \otimes V) = G(V)\otimes_Z G(V),$ as claimed. 
  \end{proof}

\begin{lemma}\label{lem:monoids}
If a C*-algebra $A$ is simple, separable, unital, nuclear, ${\mathcal Z}$-stable, and finitely generated $K_0(A)$, $K_1(A)=\{0\},$ and satisfies the UCT, then the natural map $$ Cu(A) \otimes_{alg} Cu(A)  \rightarrow Cu(A \otimes A),$$ given by
 $$([a] \otimes [b])\mapsto[a \otimes b]$$ is an isomorphism from $V(A) \otimes _{alg} V(A)$ to $V (A \otimes A)$, \textit{i.e.}
 $$V(A) \otimes_{alg} V(A) \cong V(A \otimes A).$$
\end{lemma}
\begin{proof}
We remark that the hypotheses on $A$ imply stable rank 1 by \cite{rorSR1}. Since the tensor product, $A\otimes A,$ also is simple, separable, unital, nuclear, and ${\mathcal Z}$-stable, it follows that the tensor product, $A\otimes A,$ also has stable rank one.
  Our algebra $A$ is assumed to satisfy the UCT, so then $A$ will satisfy the K\"unneth formula for tensor products in $K$-theory \cite{Schochet1982}, see also the partial counterexample due to Elliott in \cite{SchochetRosenberg} which means that we must assume finitely generated $K_0$-group:

  \begin{multline*}0 \rightarrow K_0(A) \otimes K_0(A) \oplus K_1(A) \otimes K_1(A) \rightarrow K_0(A \otimes A) \rightarrow\\ Tor(K_0(A),K_1(A)) \oplus Tor(K_1(A),K_0(A)) \rightarrow 0.\end{multline*}

 It follows from the above exact sequence that $$K_0(A) \otimes K_0(A) \rightarrow K_0(A \otimes A)$$ is an injective map. Since $K_1(A)=\{0\}$ it follows that we have an order-preserving isomorphism $$t:K_0(A) \otimes K_0(A) \rightarrow K_0(A \otimes A).$$ We will show that when restricted to the positive cones, this isomorphism becomes equivalent to the given map.
        Since $A$ has stable rank 1, there is an injective map $i\colon V(A)\rightarrow K_0 (A)$ and $V(A)$ has the cancellation property. Taking algebraic tensor products of semigroups,  we consider $V(A)\otimes_{alg} V(A).$  Taking the (semigroup) tensor product of maps,  we obtain a map $ i\otimes_{alg}i\colon    V(A)\otimes_{alg} V(A) \rightarrow  K_0(A) \otimes_{alg} K_0(A)$ where  $K_0(A) \otimes_{alg} K_0(A)$ is a semigroup tensor product of abelian groups. 
         Moreover, the semigroup tensor product of abelian groups coincides with the usual tensor product of abelian groups (by Proposition 1.4 in \cite{grillet} and the remarks after that Proposition).  We note 
that the map $ i\otimes_{alg}i$ is an injective map (using Lemma \ref{lem:tech}).   Since $A$ is stably finite and the positive cone of a tensor product of finitely generated ordered abelian groups is the tensor product of the positive cones of the ordered abelian groups, it follows that the range of the map $ i\otimes_{alg}i$ is exactly the positive cone of  $K_0(A) \otimes K_0(A).$

        Composing with the above injective map $t$, we obtain an injective map from $V(A)\otimes_{alg} V(A)$ to $K_0(A \otimes A),$ which takes an element $p \otimes_{alg} q$ to $p \otimes q.$ Since $t$ is an order isomorphism, the range of $t\circ (i\otimes_{alg} i)$ is exactly the positive cone of $K_0 (A\otimes A).$ We now observe that the map we have obtained  is in fact equal to the given map, because, as   $A \otimes A$ has stable rank 1 and is stably finite, it follows that $V(A \otimes A)$ is embedded in $K_0(A \otimes A)$ as the positive cone. It then follows that $t\circ (i\otimes_{alg} i),$ which acts on elements by taking $p\otimes_{alg} q$ to $p\otimes q,$ is in fact an injection of $V(A)\otimes_{alg} V(A)$ onto $V(A\otimes A).$ Evidently, this map coincides with the given map.

\end{proof}

\begin{remark} The argument of the above lemma can be adapted to provide a class of counter-examples to the possible surjectivity, even after closure, of the tensor product map $t\colon Cu(A)\otimes_{alg} Cu(B) \rightarrow Cu(A\otimes B).$ If an algebra is in the UCT class, and the $K$-theory groups are such that the last term in the K\"unneth sequence of \cite{Schochet1982} does not vanish, then we see that the first map in the short exact sequence
  \begin{multline*}0 \rightarrow K_0(A) \otimes K_0(A) \oplus K_1(A) \otimes K_1(A) \rightarrow K_0(A \otimes A) \rightarrow\\ Tor(K_0(A),K_1(A)) \oplus Tor(K_1(A),K_0(A)) \rightarrow 0\end{multline*}
	will not be surjective. But then, in particular, the tensor product map from $K_0(A) \otimes K_0(A)$ to $K_0 (A\otimes A)$ will not be surjective. Hence the tensor product map at the level of projection semigroups will not be surjective either, and this is an obstacle to the surjectivity of the tensor product map at the level of Cuntz semigroups (using the result that the tensor product of elements that are equivalent to a projection is an element that is equivalent to a projection).
\end{remark}
In the next theorem, the range is dense in a sense discussed in Section 4.
 \begin{theorem}
 If a C*-algebra $A$ is simple, separable, unital, stably finite, nuclear, $\mathcal{Z}$-stable, satisfies the UCT, with finitely generated $K_0(A)$, and has $K_1(A)=\{0\},$ then the natural semigroup map 
 $$t: Cu(A) \otimes_{alg} Cu(A)  \rightarrow Cu(A \otimes A)$$ given by 
$t([a]\otimes [b])=[a \otimes b]$ has dense range (under pointwise suprema).\label{th:surjectivity}
 \end{theorem}

\begin{proof} Under these hypotheses, any element $x$ in $Cu(A \otimes A)$ is either in $V(A \otimes A)$ or in $Lsc(T(A \otimes A),(0,\infty]),$ by, for example, Theorem \ref{th:representations}. 
In the case that $x$ is in  $V(A \otimes A),$ Lemma \ref{lem:monoids} shows surjectivity of $V(A)\otimes_{alg} V(A)$ onto $V(A\otimes A).$ 
The traces are a metrizable simplex \cite{EH,Brown}, so
in the case that $x$ is a continuous function in $Lsc(T(A \otimes A),(0,\infty]),$ by  Theorem 
\ref{th:surjective.function.semigroup}, there is an element of  $Lsc(T(A),(0,\infty])\otimes_{alg}Lsc(T(A),(0,\infty])$ that (uniformly) approximates $x.$ Since any lower semicontinuous function is a supremum of continuous functions, this completes the proof.
\end{proof}
\begin{remark} We mention that the above result applies in the finite-dimensional case, in other words, the case of matrix algebras, and in this case, dense range is equivalent to surjectivity.\end{remark}

\section{Deducing surjectivity results}

To obtain surjectivity results in greater generality, some form of completion operation must be introduced, and we do this next.

The natural topology on the projection-class elements of the Cuntz semigroup is the discrete topology. Thus, we were able to handle projection-class elements without explicit reference to a topology. However, when considering the tensor product of purely positive elements, it appears inevitable that we must consider topologies, which means that we must consider completions of the algebraic tensor product with respect to a specified topology.

 Our approach is based on a very general construction of the so-called inductive (or injective) tensor product of topological vector spaces, due to   Gro\-then\-dieck, see \cite[d\'efinition 3]{grot}. Thus, we form the algebraic tensor product of abelian semigroups, see \cite{grillet}, and then view  elements of $Cu(A)\otimes_{alg} Cu(B)$ as   functions on $D(A)\times D(B),$ where $D(A)$ and $D(B)$ denote the dimension functions on the Cuntz semigroup(s). The inductive topology is the (initial) topology induced by this embedding. 
We then take the topological completion of $Cu(A)\otimes_{alg} Cu(B)$ with respect to this topology (see \cite[ex. 6.L, and pp. 195-6]{kelley} for information on completions). As previously mentioned, we only need to perform the above construction on the set of those elements whose image under the tensor product map is purely positive: this set will be further  characterized later. For brevity we refer to these elements as the purely positive elements. We have that for the purely positive elements of
 $Cu(A)\otimes_{alg} Cu(B)$, an increasing sequence $x_n$ converges (pointwise) to an element of the completion if and only if $(d_1, d_2)(x_n)$ converges for all $d_1\in D(A)$ and $d_2 \in D(B).$ The limit of the sequence exists in the completion, and we define the inductive tensor product of Cuntz semigroups, denoted $Cu(A) \otimes Cu(B),$ to be $Cu(A)\otimes_{alg} Cu(B)$ augumented by the set of all such limits.  

We also have the minimal embedding tensor product, given as follows. Consider the natural tensor product map $t\colon Cu(A)\otimes_{alg} Cu(B)\rightarrow Cu(A\otimes_{min} B).$  This map induces a uniformizable topology on its domain (sometimes called the inital topology). Taking, then, the  completion of the domain with respect to this topology, and proceeding as in the previous paragraph, we obtain the minimal embedding tensor product, $Cu(A)\otimes_{min} Cu(B).$ If, in the above, we replace   $Cu(A\otimes_{min} B)$ by $Cu(A\otimes_{max} B),$ then we obtain the maximal embedding tensor product, denoted $Cu(A)\otimes_{max} Cu(B).$ 


\begin{prop} If the dimension functions are determined by traces, and if the normalized traces  $T(A\otimes_{min} B)$ are isomorphic to $T(A)\times T(B),$ then the inductive tensor product $Cu(A)\otimes Cu(B)$ coincides with the minimal embedding tensor product $Cu(A)\otimes_{min} Cu(B).$ \label{prop:same.topologies}\end{prop}
\begin{proof} 
The inductive tensor product is given by adding to the algebraic tensor product the limits of increasing sequences $(x_n)$ that are such that  $(d_1 \times d_2)(x_n)$ converges  for each $d_1\in D(A)$ and $d_2 \in D(B).$ (We need only consider purely positive elements.)
On the other hand, the minimal embedding tensor product is given by adding to the algebraic tensor product the limits of increasing sequences $(x_n)$ that are such that $d(t(x_n))$ converges for each $d \in D(A\otimes_{min} B),$ again for purely positive elements.
The hypothesis implies that $D(A\otimes_{min} B)$ is isomorphic to $D(A)\times D(B),$ but then the two constructions considered above will coincide.   
\end{proof}


The results of the previous sections are  valid  within the class of algebra where the purely positive elements of $Cu(A \otimes A)$ are isomorphic to the set of lower semicontinuous, affine, strictly positive functions on the tracial state space of the unital C*-algebra $A.$ 

In general, of course, since we regard $Cu(A)$ as being topologically a disjoint union of two sets, namely the set of elements having the same class as a projection, and the set of purely positive elements, the tensor product of Cuntz semigroups will \textit{a priori} be a union of four sets. As explained above, the topological issues only arise when considering the component $Cu(A)|_{\mbox{\tiny pure}}\otimes Cu(B)|_{\mbox{\tiny pure}}.$ 
In general, when considering the tensor product map $t\colon Cu(A) \otimes Cu(A) \rightarrow Cu(A \otimes A)$ the following three cases  appear: \\
 \textbf{Case 1:} projection elements tensored with projection elements,\\
 \textbf{Case 2:} purely positive tensored with purely positive elements, and\\
\textbf{ Case 3:} projection elements tensored with purely positive elements.
We now summarize our results on algebraic tensor products, for each of these three cases, in the setting of completed tensor products.
\begin{theorem} Suppose that $A$ is a  C*-algebra which is unital, simple, separable, stably finite, nuclear, $\mathcal{Z}$-stable, satisfies the UCT, with finitely generated $K_0(A)$, has $K_1(A)=\{0\}$ and  satisfies the Blackadar--Handelman conjecture. Then each of the maps \begin{eqnarray*}
t:Cu(A)|_{\mbox{\tiny pure}} \otimes Cu(A)|_{\mbox{\tiny pure}} \rightarrow Cu(A \otimes A)\\
t:Cu(A)|_{\mbox{\tiny pure}} \otimes Cu(A)|_{\mbox{\tiny proj}} \rightarrow Cu(A \otimes A)\\
t:Cu(A)|_{\mbox{\tiny proj}}  \otimes Cu(A)|_{\mbox{\tiny pure}} \rightarrow Cu(A \otimes A)\\
t:Cu(A)|_{\mbox{\tiny proj}}  \otimes Cu(A)|_{\mbox{\tiny proj}}  \rightarrow Cu(A \otimes A)\\
\end{eqnarray*}
is injective. The first three of these maps have range contained in the purely positive elements of $Cu(A \otimes A),$ the last map has range contained in the projection-class elements of $Cu(A \otimes A).$ The domains are injective tensor products of Cuntz semigroups. The first and last of the above maps are surjective onto their range. 
\label{th:4components} 
\end{theorem}    
\begin{proof} We first show that the behaviour with respect to purely positive and projection-class elements is as claimed. A positive element has the class of a projection if and only if the spectrum of the element has a spectral gap at zero, and the spectrum of $a\otimes b$ is given by the set of all pairwise products $\{\lambda\mu\, |\, \lambda\in \mbox{Sp}(a),  \mu\in \mbox{Sp}(b) \}.$ It can thus be seen that $a\otimes b$ has the class of a projection if and only if both $a$ and $b$ have the class of a projection. 

We now consider surjectivity. Our Lemma \ref{lem:monoids} gives surjectivity onto the projection elements of $Cu(A \otimes A),$ and 
Lemma \ref{lem:pou.affine} provides an approximation (from below) of purely positive elements of $Cu(A)\otimes Cu(A)$ that are continuous when viewed as functions  on $D(A)\times D(A).$ More precisely,  the image of $Cu(A) \otimes_{{alg}} Cu(A)$ is dense, with respect to the supremum norm over the (compact) space $D(A)\times D(A),$ within the continuous positive affine functions on   $D(A)\times D(A).$ It follows that after taking completions, the range of the tensor product map $t\colon Cu(A)\otimes Cu(A) \rightarrow Cu(A\otimes A)$ contains at least the  purely positive elements corresponding to continuous functions.
Since any lower semicontinuous function is a supremum of continuous functions, this implies surjectivity.

We now consider the injectivity of these maps. Suppose that $x, x' \in Cu(A) \otimes Cu(A)$ are such that $t(x)=t(x').$ Consider first the case where $t(x)$ and $t(x')$ belong to the purely positive part of $Cu(A) \otimes Cu(A).$ We have that $d \circ t(x)=d\circ t(x')$ for all dimension functions $d$  on $A \otimes A$.  Choose a dimension function $d$ on $A \otimes A$ that comes from a tensor product $\tau_1 \otimes \tau_2$ of traces on $A$. Thus, if $t$ is the map taking $[a] \otimes [b]$ to $[a \otimes b]$, and $d(t(x))=d(t(x'))$, we deduce that $(d_1 \otimes d_2)(x)=(d_1 \otimes d_2)(x')$ where $d_1$ and $d_2$ are the dimension states on $Cu(A)$ that come from $\tau_1$ and $\tau_2$ respectively. Since $A$ satisfies the Blackadar-Handelman conjecture, this means that $d_1$ and $d_2$ can be chosen to be  equal to any  two dimension states of $A.$  But this means that $x$ and $x'$ are equal in the tensor product $Cu(A) \otimes Cu(A).$ The two mixed cases are similar. The case of projections follows from the fact that the natural map from $V(A)\otimes V(A)$ to  $V(A\otimes A)$ is an injection, by Lemma \ref{lem:monoids}.
\end{proof}

Combining the above with Theorem \ref{th:surjectivity}, we have the corollary:
\begin{corollary} If an unital C*-algebra $A$ is simple, separable, stably finite, nuclear, $\mathcal{Z}$-stable, satisfies the UCT, with finitely generated $K_0(A)$, has $K_1(A)=\{0\}$ then the natural tensor product map 
 $$t: Cu(A) \otimes Cu(A)  \rightarrow Cu(A \otimes A)$$ given by 
$t([a],[b])=[a \otimes b]$ is surjective, and becomes an isomorphism when restricted to 
$$Cu(A)|_{\mbox{\tiny pure}} \otimes Cu(A)|_{\mbox{\tiny pure}} \sqcup Cu(A)|_{\mbox{\tiny proj}}  \otimes Cu(A)|_{\mbox{\tiny proj}}.$$ 
\end{corollary} 
 
Thus the mixed elements of the tensor product are evidently an obstacle to injectivity of the unrestricted tensor product map. This is because, by Theorem \ref{th:surjectivity}, every purely positive element of  $Cu(A \otimes A)$ is the image of some element of $Cu(A)|_{\mbox{\tiny pure}} \otimes Cu(A)|_{\mbox{\tiny pure}},$  but on the other hand, as shown in the proof of Theorem \ref{th:4components}, the image of a purely positive element tensored with a projection-class element is a purely positive element. It thus  follows that in this case the tensor product map is not injective. Therefore, it is quite rare for the tensor product map to be an isomorphism. 

We note the following further corollary:
 \begin{corollary} If an unital C*-algebra $A$ is simple, stably projectionless,  stably finite, nuclear, $\mathcal{Z}$-stable, satisfies the UCT, with finitely generated $K_0(A)$, and has $K_1(A)=\{0\}$  then the natural tensor product map 
 $$t: Cu(A) \otimes Cu(A)  \rightarrow Cu(A \otimes A)$$ given by 
$t([a]\otimes [b])=[a \otimes b]$ is an isomorphism. 
\end{corollary}
A class of algebras to which our results can be applied is the class of simple AT-algebras with trivial $K_1$-group (or, equivalently, AI algebras). 
 \begin{corollary} If $A$ is a simple AI-algebra with finitely generated K-theory, then the natural tensor product map 
 $$t: Cu(A) \otimes Cu(A)  \rightarrow Cu(A \otimes A)$$ given by 
$t([a],[b])=[a \otimes b]$ is surjective, and becomes an isomorphism when restricted to 
$$Cu(A)|_{\mbox{\tiny pure}} \otimes Cu(A)|_{\mbox{\tiny pure}} \sqcup Cu(A)|_{\mbox{\tiny proj}}  \otimes Cu(A)|_{\mbox{\tiny proj}}.$$ 
 \end{corollary}

 Dini's theorem shows that over a compact base space, pointwise convergence of increasing sequences of continuous functions implies uniform convergence.  It follows that in the simple special case of C*-algebras having Cuntz semigroups isomorphic to the positive lower semicontinuous functions on a compact space, the rather coarse topology provided by the definition of the injective tensor product is equivalent to a much finer topology. Thus, for the class of algebras under consideration, we expect that most of the ways to construct the tensor product will coincide.
The tensor product of  \cite{APT14} involves both a joint continuity condition  and a separate continuity condition. (See page 7 of \cite{APT14} for the precise statement of both conditions.) In the closely related setting of tensor products of topological vector spaces, a tensor product involving both a joint continuity condition  and a separate continuity condition would lie strictly between the inductive tensor product and the projective tensor product. We have already shown (see Proposition \ref{prop:same.topologies}) that in our setting, the equivalent of the inductive tensor product and the projective tensor product happen to coincide. At present, however, we do not have a good topological description of the tensor product of  \cite{APT14}.
 

\begin{remark} 
The tensor product of two abelian semigroups is constructed by forming a free abelian semigroup and passing to the quotient in which $(a + a')\otimes b=(a\otimes b)+(a'\otimes b)$ and $(a\otimes b')+(a \otimes b)=a\otimes (b+b')$.
 One can equivalently define the tensor product as an abstract abelian semigroup on which additive maps are given by biadditive maps on the original semigroups.   However, Grillet in \cite{grillet}, Theorem 2.1, shows that, as for any universal property definition, the tensor product of abelian semigroups is unique up to semigroup isomorphism. 
\end{remark}

We recall that the stable rank one property was used in a fundamental way in our proofs, to, for example, simplify the compact containment property. (We further used the stable rank one property to obtain cancellation in the projection semigroup, and to ensure that projection class and purely positive elements can be distinguished by a spectral criterion.)  Investigating the exact relationship between the several possible tensor products suggests venturing beyond the stable rank one case; in which case additional issues arise that are beyond the scope of this paper, but seem to be a reasonable question for  future investigation.

\medskip
\end{document}